\documentclass[12pt]{amsart}  
\usepackage{amssymb}
\usepackage{amscd}

\newtheorem{thm}{Theorem}[section]
\newtheorem{cor}[thm]{Corollary}

\newtheorem{lemma}[thm]{Lemma}

\newtheorem*{thm1}{Theorem 6.1}

\theoremstyle{remark}
\newtheorem{remark}[thm]{Remark}

\newtheorem*{notation}{Notation}

\theoremstyle{definition}
\newtheorem{definition}[thm]{Definition}

\numberwithin{equation}{subsection}

\newcommand{\Hom}{\operatorname{Hom}}

\newcommand{\id}{{\mathtt{Id}}}

\newcommand{\shHom}{\underline{\operatorname{Hom}}}

\newcommand{\Symm}{\operatorname{Sym}}

\newcommand{\pr}{{\mathtt{pr}}}

\newcommand{\DR}{\mathtt{DR}}

\newcommand{\vac}{{\mathbf{1}}}
\newcommand\dual[1]{{#1}^{\vee}}
\newcommand{\ip}{{\langle\ ,\ \rangle}}

\newcommand{\coComm}{\operatorname{coComm}}

\newcommand{\MC}{\operatorname{MC}}

\newcommand{\Def}{\operatorname{Def}}

\newcommand{\Der}{{\mathtt{Der}}}

\newcommand{\op}{\mathtt{op}}
\DeclareMathOperator{\sgn}{sgn}
\newcommand{\dR}{\widehat{\mathrm{d}}_\mathrm{d\!\! R}}
\newcommand{\bc}{\mathtt{B}}

\makeatletter
\renewcommand{\subsubsection}{\@startsection
{subsubsection}%
{2}%
{0mm}%
{-\baselineskip}%
{-0.5\baselineskip}%
{\normalfont\normalsize\bfseries }}%
\makeatother

\begin{document}

\title{Formality for algebroid stacks}

\author[P.Bressler]{Paul Bressler}
\address{Max-Planck-Institut f\"{u}r Mathematik,
Vivatsgasse 7, 53111 Bonn, Germany} \email{paul.bressler@gmail.com}

\author[A.Gorokhovsky]{Alexander Gorokhovsky}
\address{Department of Mathematics, UCB 395,
University of Colorado, Boulder, CO~80309-0395, USA} \email{Alexander.Gorokhovsky@colorado.edu}

\author[R.Nest]{Ryszard Nest}
\address{Department of Mathematics,
Copenhagen University, Universitetsparken 5, 2100 Copenhagen, Denmark}
 \email{rnest@math.ku.dk}

\author[B.Tsygan]{Boris Tsygan}
\address{Department of
Mathematics, Northwestern University, Evanston, IL 60208-2730, USA} \email{tsygan@math.northwestern.edu}

\begin{abstract}
We extend the formality theorem of M.~Kontsevich from deformations of the structure sheaf on a manifold to deformations
of gerbes.
\end{abstract}

\thanks{
A. Gorokhovsky was partially supported by NSF grant DMS-0400342. B. Tsygan was partially supported by NSF grant
DMS-0605030}
\maketitle

\section{Introduction}
In the fundamental paper \cite{K} M.~Kontsevich showed that  the set of equivalence classes of formal deformations the algebra of functions on a manifold is in one-to-one correspondence with the set of equivalence classes of formal Poisson structures on the manifold. This result was obtained as a corollary of the formality of the Hochschild complex of the algebra of functions on the manifold conjectured by M.~Kontsevich  (cf. \cite{K1}) and proven
in \cite{K}. Later proofs by a different method were given in \cite{T} and in \cite{DTT}.

In this paper we extend the formality theorem of M.~Kontsevich to deformations of
gerbes on smooth manifolds, using the method of \cite{DTT}. Let $X$ be a smooth manifold; we denote by $\mathcal{O}_X$  the sheaf of complex valued $C^\infty$ functions on $X$. For a twisted form $\mathcal{S}$ of $\mathcal{O}_X$ regarded as an algebroid stack (see Section
\ref{subsection: twisted forms}) we denote by  $[\mathcal{S}]_{dR} \in H^3_{dR}(X)$ the de Rham class of $\mathcal{S}$. The main result of this paper establishes an equivalence of $2$-groupoid valued functors of Artin $\mathbb{C}$-algebras between $\Def(\mathcal{S})$ (the formal deformation theory of $\mathcal{S}$, see \cite{BGNT1}) and the Deligne $2$-groupoid of Maurer-Cartan elements of $L_{\infty}$-algebra of multivector
fields on $X$ twisted by a closed three-form representing $[\mathcal{S}]_{dR}$:
\begin{thm1}\label{main thm}
Suppose that $\mathcal{S}$ is a twisted form of $\mathcal{O}_X$. Let $H$ be a closed $3$-form on $X$ which represents
$[\mathcal{S}]_{dR} \in H^3_{dR}(X)$. For any Artin algebra $R$ with maximal ideal $\mathfrak{m}_R$ there is an
equivalence of $2$-groupoids
\[
\MC^2(\mathfrak{s}(\mathcal{O}_X)_H\otimes\mathfrak{m}_R) \cong \Def(\mathcal{S})(R)
\]
natural in $R$.
\end{thm1}

Here, $\mathfrak{s}(\mathcal{O}_X)_H$ denotes the $L_{\infty}$-algebra of multivector fields with the trivial differential, the binary operation given by Schouten bracket, the ternary operation given by $H$ (see \ref{subsection: structures on multivectors}) and all other operations equal to zero. As a corollary  of this result  we  obtain that the  isomorphism classes of formal deformations of $\mathcal S$ are in a bijective correspondence with equivalence classes of the formal \emph{twisted Poisson structures} defined by P.~Severa
and A.~Weinstein in \cite{SW}.

The proof of the Theorem proceeds along the following lines. As a starting point we use the construction of the Differential Graded Lie Algebra (DGLA) controlling the deformations of $\mathcal{S}$. This construction was obtained in \cite{BGNT, BGNT1}. Next we construct a chain of $L_{\infty}$-quasi-isomorphisms between this DGLA and $\mathfrak{s}(\mathcal{O}_X)_H$, using the techniques of \cite{DTT}. Since $L_{\infty}$-quasi-isomorphisms induce equivalences of respective Deligne groupoids, the result follows.

The paper is organized as follows. Section \ref{prelim} contains the preliminary material on jets and deformations. Section \ref{defalgstack} describes the results on the deformations of algebroid stacks. Section \ref{review of DTT} is a short exposition of \cite{DTT}.
Section \ref{for} contains the main technical result of the paper: the construction of the chain of quasi-isomorphisms mentioned above. Finally, in Section \ref{deform} the main
theorem is deduced from the results of Section \ref{for}.

The paper was written while the first author was visiting Max-Plank-Institut f\"ur Mathematik, Bonn.

\section{Preliminaries}\label{prelim}
\subsection{Notations}
Throughout this  paper, unless specified otherwise, $X$ will denote a $C^{\infty}$ manifold. By $\mathcal{O}_X$ we
denote the sheaf of complex-valued $C^{\infty}$ functions on $X$. $\mathcal{A}^{\bullet}_X$ denotes the sheaf of
differential forms on $X$, and $\mathcal{T}_X$ the sheaf of vector fields on $X$. For a ring $K$ we denote by
$K^{\times}$ the group of invertible elements of $K$.
\subsection{Jets}\label{subsection: jets}
Let $\pr_i\colon X\times X\to X$, $i = 1,2$, denote the projection on the $i^{\text{th}}$ factor. Let $\Delta_X \colon
X\to X\times X$ denote the diagonal embedding. Let $\mathcal{I}_{X} := \ker(\Delta_X^*)$.

For a locally-free ${\mathcal O}_{X }$-module of finite rank ${\mathcal E}$ let
\begin{eqnarray*}
\mathcal{J}_{X }^k({\mathcal E}) & := &
(\pr_1)_*\left({\mathcal O}_{X\times X}/{\mathcal
I}_{X}^{k+1}\otimes_{\pr_2^{-1}{\mathcal
O}_{X}}\pr_2^{-1}{\mathcal E}\right) \ , \\
\mathcal{J}^k_{X}& := & \mathcal{J}_{X }^k(\mathcal{O}_{X }) \
.
\end{eqnarray*}
It is clear from the above definition that $\mathcal{J}^k_{X }$ is, in a natural way, a commutative algebra and
$\mathcal {J}_{X }^k({\mathcal E})$ is a $\mathcal {J}^k_{X}$-module.

Let

\[
\vac^{(k)} \colon  \mathcal{O}_{X }\to \mathcal{J}^k_{X }
\]
denote the composition
\[
\mathcal{O}_{X } \xrightarrow{\pr_1^*} (\pr_1)_*\mathcal{O}_{X\times
X  } \to \mathcal{J}^k_{X }
\]
In what follows, unless stated explicitly otherwise, we regard $\mathcal{J}_{X }^k({\mathcal E})$ as a $\mathcal
{O}_{X}$-module via the map $\vac^{(k)}$.

Let
\[
j^k\colon  \mathcal{E} \to \mathcal{J}_{X }^k(\mathcal{E})
\]
denote the composition
\[
\mathcal{E} \xrightarrow{e\mapsto 1\otimes e}
(\pr_1)_*\mathcal{O}_{X\times X  } \otimes_\mathbb{C} \mathcal{E}
\to \mathcal{J}_{X }^k(\mathcal{E})
\]
The map $j^k$ is not $\mathcal{O}_{X }$-linear unless $k=0$.

For $0\leq k\leq l$ the inclusion ${\mathcal I}_{X }^{l+1}\to{\mathcal I}_{X}^{k+1}$ induces the surjective map
$\pi_{l,k}\colon {\mathcal J}^l_{X }({\mathcal E}) \to {\mathcal J}^k_{X }({\mathcal E})$. The sheaves ${\mathcal
J}^k_{X }({\mathcal E})$, $k=0,1,\ldots$ together with the maps $\pi_{l,k}$, $k\leq l$ form an inverse system. Let
${\mathcal J}_{X }({\mathcal E}) = {\mathcal J}^\infty_{X }({\mathcal E}) :=
\varprojlim {\mathcal J}^k_{X }({\mathcal E})$. Thus, ${\mathcal J}_{X }({\mathcal E})$ carries a
natural topology.

The maps $\vac^{(k)}$ (respectively, $j^k$), $k=0,1,2,\ldots$ are compatible with the projections $\pi_{l,k}$, i.e.
$\pi_{l,k}\circ\vac^{(l)} = \vac^{(k)}$ (respectively, $\pi_{l,k}\circ j^l = j^k$). Let $\vac := \varprojlim
\vac^{(k)}$,
 $j^\infty := \varprojlim j^k$.

Let
\begin{multline*}
d_1 \colon  {\mathcal O}_{{X\times X} }\otimes_{\pr_2^{-1}{\mathcal O}_{X}}\pr_2^{-1}{\mathcal E} \longrightarrow \\
\pr_1^{-1}\mathcal{A}^1_{X}\otimes_{\pr_1^{-1}{\mathcal O}_{X}}{\mathcal O}_{{X\times X} }\otimes_{\pr_2^{-1}{\mathcal
O}_{X}}\pr_2^{-1}{\mathcal E}
\end{multline*}
denote the exterior derivative along the first factor. It satisfies
\[
d_1({\mathcal I}_{X}^{k+1}\otimes_{\pr_2^{-1}{\mathcal
O}_{X}}\pr_2^{-1}{\mathcal E})\subset
\pr_1^{-1}\mathcal{A}^1_X\otimes_{\pr_1^{-1}{\mathcal O}_{X}}{\mathcal
I}_{X}^k\otimes_{\pr_2^{-1}{\mathcal O}_{X}}\pr_2^{-1}{\mathcal E}
\]
for each $k$ and, therefore, induces the map
\[
d_1^{(k)} \colon  {\mathcal J}^k({\mathcal
E})\to\mathcal{A}^1_{X}\otimes_{{\mathcal O}_{X}}{\mathcal
J}^{k-1}({\mathcal E})
\]
The maps $d_1^{(k)}$ for different values of $k$ are compatible with the maps $\pi_{l,k}$ giving rise to the
\emph{canonical flat connection}
\[
\nabla^{can} \colon  {\mathcal J}_{X}({\mathcal
E})\to\mathcal{A}^1_{X}\otimes_{{\mathcal O}_{X}}{\mathcal
J}_{X}({\mathcal E}) \ .
\]

\subsection{Deligne groupoids}
In \cite{Del} P.~Deligne and, independently, E.~Getzler in \cite{G2} associated to a nilpotent DGLA $\mathfrak{g}$ concentrated in degrees grater than or equal to $-1$ the $2$-groupoid, referred to as \emph{the Deligne $2$-groupoid} and denoted $\MC^2(\mathfrak{g})$ in \cite{BGNT}, \cite{BGNT1} and below. The objects of $\MC^2(\mathfrak{g})$ are the Maurer-Cartan elements of $\mathfrak{g}$. We refer the reader to \cite{G2} (as well as to \cite{BGNT1}) for a detailed description. The above notion was extended and generalized by E.~Getzler in \cite{G1} as follows.

To a nilpotent $L_{\infty}$-algebra $\mathfrak{g}$ Getzler associates a (Kan) simplicial set $\gamma_{\bullet}(\mathfrak{g})$ which is functorial for $L_\infty$ morphisms. If $\mathfrak{g}$ is concentrated in degrees greater than or equal to $1-l$, then the simplicial set $\gamma_{\bullet}(\mathfrak{g})$ is an $l$-dimensional hypergroupoid in the sense of J.W.~Duskin (see \cite{duskin}) by \cite{G1}, Theorem 5.4.

Suppose that $\mathfrak{g}$ is a nilpotent $L_{\infty}$-algebra concentrated in degrees grater than or equal to $-1$. Then, according to \cite{duskin}, Theorem 8.6 the simplicial set $\gamma_{\bullet}(\mathfrak{g})$ is the nerve of a bigroupoid, or, a $2$-groupoid in our terminology. If $\mathfrak{g}$ is a DGLA concentrated in degrees grater than or equal to $-1$ this $2$-groupoid coincides with $\MC^2(\mathfrak{g})$ of Deligne and Getzler alluded to earlier. We extend our notation to the more general setting of nilpotent $L_{\infty}$-algebras as above and denote by $\MC^2(\mathfrak{g})$ the $2$-groupoid furnished by \cite{duskin}, Theorem 8.6.

For an $L_{\infty}$-algebra $\mathfrak{g}$ and a nilpotent commutative algebra $\mathfrak{m}$ the $L_{\infty}$-algebra $\mathfrak{g}\otimes\mathfrak{m}$ is nilpotent, hence the simplicial set $\gamma_{\bullet}(\mathfrak{g}\otimes \mathfrak{m})$ is defined and enjoys the following homotopy invariance property (\cite{G1}, Proposition 4.9, Corollary 5.11):

\begin{thm}\label{thm: quism invariance of mc}
Suppose that $f \colon \mathfrak{g}\to\mathfrak{h}$ is a
quasi-isomorphism of $L_{\infty}$ algebras and  let $\mathfrak{m}$
be a nilpotent commutative algebra. Then the induced map
\[
\gamma_{\bullet}(f\otimes \id) \colon
\gamma_{\bullet}(\mathfrak{g}\otimes \mathfrak{m}) \to
\gamma_{\bullet}(\mathfrak{h}\otimes \mathfrak{m})
\]
is a homotopy equivalence.
\end{thm}

\subsection{Algebroid stacks}\label{subsection: algebroid stacks}
Here we give a very brief overview, referring the reader to \cite{DAP, KS} for the details. Let $k$ be a field of
characteristic zero, and let $R$ be a commutative $k$-algebra.
\begin{definition}
A stack in $R$-linear categories $\mathcal{C}$ on $X$ is an \emph{$R$-algebroid stack} if it is locally nonempty and
locally connected, i.e. satisfies
\begin{enumerate}
\item any point $x\in X$ has a neighborhood $U$ such that $\mathcal{C }(U)$ is nonempty;

\item for any $U\subseteq X$, $x\in U$, $A, B\in\mathcal{C}(U)$ there exits a neighborhood $V\subseteq U$ of $x$
    and an isomorphism $A\vert_V\cong B\vert_V$.
\end{enumerate}
\end{definition}

For a prestack $\mathcal{C}$ we denote by $\widetilde{\mathcal{C}}$ the associated stack.

For a category $C$  denote by $iC$ the subcategory of isomorphisms in $C$; equivalently, $iC$ is the maximal
subgroupoid in $C$. If $\mathcal{C}$ is an algebroid stack then  the stack associated to the substack of isomorphisms
$\widetilde{i\mathcal{C}}$ is a gerbe.

 For an algebra $K$ we denote by $K^+$ the linear category with a single object whose endomorphism algebra is $K$.
For a sheaf of algebras $\mathcal{K}$ on $X$ we denote by $\mathcal{K}^+$ the prestack in linear categories given by $U
\mapsto \mathcal{K}(U)^+$. Let $\widetilde{\mathcal{K}^+}$ denote the associated stack. Then,
$\widetilde{\mathcal{K}^+}$ is an algebroid stack equivalent to the stack of locally free $\mathcal{K}^\op$-modules of
rank one.


By a \emph{twisted form of $\mathcal{K}$} we mean an
algebroid stack locally equivalent to $\widetilde{\mathcal{K}^+}$. It is easy to see that the equivalence classes of
twisted forms of $\mathcal{K}$ are bijective correspondence with $H^2(X;\mathtt{Z}(\mathcal{K})^\times)$, where
$\mathtt{Z}(\mathcal{K})$ denotes the center of $\mathcal{K}$.

\subsection{Twisted forms of $\mathcal{O}$}\label{subsection: twisted forms}

Twisted forms of $\mathcal{O}_X$ are in bijective correspondence with $\mathcal{O}_X^\times$-gerbes: if $\mathcal{S}$
is a twisted form of $\mathcal{O}_X$, the corresponding gerbe is the substack $i\mathcal{S}$ of isomorphisms in
$\mathcal{S}$. We shall not make a distinction between the two notions.

The equivalence classes of twisted forms of $\mathcal{O}_X$ are in bijective correspondence with
$H^2(X;\mathcal{O}_X^\times)$. The composition
\[
\mathcal{O}_X^\times \to \mathcal{O}_X^\times / \mathbb{C}^\times \xrightarrow{\log} \mathcal{O}_X/\mathbb{C} \xrightarrow{j^\infty} \DR(\mathcal{J}_X/\mathcal{O}_X)
\]
induces the map $H^2(X;\mathcal{O}_X^\times) \to H^2(X;\DR(\mathcal{J}_X/\mathcal{O}_X)) \cong
H^2(\Gamma(X;\mathcal{A}^\bullet_X \otimes \mathcal{J}_X/\mathcal{O}_X), \nabla^{can})$. We denote by $[\mathcal{S}]$
the image in the latter space of the class of $\mathcal{S}$.

The short exact sequence
\[
0 \to \mathcal{O}_X \xrightarrow{\vac} \mathcal{J}_X \to \mathcal{J}_X/\mathcal{O}_X \to 0
\]
gives rise to the short exact sequence of complexes
\[
0 \to \Gamma(X;\mathcal{A}^\bullet_X) \to \Gamma(X;\DR(\mathcal{J}_X)) \to \Gamma(X; \DR(\mathcal{J}_X/\mathcal{O}_X)) \to 0 ,
\]
hence to the map (connecting homomorphism) $H^2(X;\DR(\mathcal{J}_X/\mathcal{O}_X)) \to H^3_{dR}(X)$. Namely, if $B \in
\Gamma(X; \mathcal{A}^2_X\otimes\mathcal{J}_X)$ maps to $\overline{B} \in \Gamma(X;
\mathcal{A}^2_X\otimes\mathcal{J}_X/\mathcal{O}_X)$ which represents $[\mathcal{S}]$, then there exists a unique $H \in
\Gamma(X; \mathcal{A}^3)$ such that $\nabla^{can}B = \DR(\vac)(H)$. The form $H$ is closed and represents the image of
the class of $\overline{B}$ under the connecting homomorphism.

\begin{notation}
We denote by $[\mathcal{S}]_{dR}$ the image of $[\mathcal{S}]$ under the map \[H^2(X;\DR(\mathcal{J}_X/\mathcal{O}_X))
\to H^3_{dR}(X). \]
\end{notation}

\section{Deformations of algebroid stacks}\label{defalgstack}

\subsection{Deformations of linear stacks}
Here we describe the notion of $2$-groupoid of deformations   of an algebroid stack. We follow \cite{BGNT1} and refer
the reader to that paper for   all the  proofs and   additional details.

 For an $R$-linear category $\mathcal{C}$ and homomorphism of algebras $R\to S$ we denote by
$\mathcal{C}\otimes_R S$ the category with the same objects as $\mathcal{C}$ and morphisms defined by
$\Hom_{\mathcal{C}\otimes_R S}(A,B) = \Hom_\mathcal{C}(A,B)\otimes_R S$.

For a prestack $\mathcal{C}$ in $R$-linear categories we denote by $\mathcal{C}\otimes_R S$ the prestack associated to
the fibered category $U\mapsto\mathcal{C}(U)\otimes_R S$.

\begin{lemma}[\cite{BGNT1}, Lemma 4.13]
Suppose that $\mathcal{A}$ is a sheaf of $R$-algebras and $\mathcal{C}$ is an $R$-algebroid stack. Then
$\widetilde{\mathcal{C}\otimes_R S}$ is an algebroid stack.
\end{lemma}

 \
Suppose now that $\mathcal{C}$ is a stack in $k$-linear categories on $X$ and $R$ is a commutative Artin $k$-algebra.
We denote by $\Def(\mathcal{C})(R)$ the $2$-category with
\begin{itemize}
\item objects: pairs $(\mathcal{B}, \varpi)$, where $\mathcal{B}$ is a stack in $R$-linear categories flat over $R$
    and $\varpi : \widetilde{\mathcal{B}\otimes_R k} \to \mathcal{C}$ is an equivalence of stacks in $k$-linear
    categories

\item $1$-morphisms: a $1$-morphism $(\mathcal{B}^{(1)}, \varpi^{(1)})\to (\mathcal{B}^{(2)}, \varpi^{(2)})$ is a
    pair $(F,\theta)$ where $F : \mathcal{B}^{(1)}\to \mathcal{B}^{(2)}$ is a $R$-linear functor and $\theta :
    \varpi^{(2)}\circ (F\otimes_R k) \to \varpi^{(1)}$ is an isomorphism of functors

\item $2$-morphisms: a $2$-morphism $(F',\theta')\to (F'',\theta'')$ is a morphism of $R$-linear functors $\kappa :
    F'\to F''$ such that $\theta''\circ(\id_{\varpi^{(2)}}\otimes(\kappa\otimes_R k)) = \theta'$
\end{itemize}

The $2$-category $\Def(\mathcal{C})(R)$ is a $2$-groupoid.

Let   $\mathcal{B}$ be a prestack on $X$ in $R$-linear categories. We say that  $\mathcal{B}$ is  \emph{flat} if for
any $U\subseteq X$, $A,B\in\mathcal{B}(U)$ the sheaf $\shHom_\mathcal{B}(A,B)$ is flat (as a sheaf of $R$-modules).
\begin{lemma}[\cite{BGNT1}, Lemma 6.2]\label{lemma: def of algd is algd}
Suppose that $\mathcal{B}$ is a flat $R$-linear stack on $X$ such that $\widetilde{\mathcal{B}\otimes_R k}$ is an
algebroid stack. Then $\mathcal{B}$ is an algebroid stack.
\end{lemma}

\subsection{Deformations of twisted forms of $\mathcal{O}$}\label{subsection: deformations of twisted forms}
Suppose that $\mathcal{S}$ is a twisted form of $\mathcal{O}_X$. We will now describe the DGLA controlling the
deformations of $\mathcal{S}$.




The complex $\Gamma(X;\DR(C^\bullet(\mathcal{J}_X))
= (\Gamma(X;\mathcal{A}^\bullet_X \otimes C^\bullet(\mathcal{J}_X)), \nabla^{can} + \delta)$ is a differential graded
brace algebra in a canonical way. The abelian Lie algebra $\mathcal{J}_X = C^0(\mathcal{J}_X)$ acts on the brace
algebra $C^\bullet(\mathcal{J}_X)$ by derivations of degree $-1$ by Gerstenhaber bracket. The above action factors
through an action of $\mathcal{J}_X/\mathcal{O}_X$. Therefore, the abelian Lie algebra
$\Gamma(X;\mathcal{A}^2_X\otimes\mathcal{J}_X/\mathcal{O}_X)$ acts on the brace algebra $\mathcal{A}^\bullet_X \otimes
C^\bullet(\mathcal{J}_X)$ by derivations of degree $+1$. Following longstanding tradition, the action of an element $a$
is denoted by $i_a$.

Due to commutativity of $\mathcal{J}_X$, for any $\omega\in
\Gamma(X;\mathcal{A}^2_X\otimes\mathcal{J}_X/\mathcal{O}_X)$ the operation $\iota_\omega$ commutes with the Hochschild
differential $\delta$. If, moreover, $\omega$ satisfies $\nabla^{can}\omega = 0$, then $\nabla^{can} + \delta +
i_\omega$ is a square-zero derivation of degree one of the brace structure. We refer to the complex
\[
\Gamma(X;\DR(C^\bullet(\mathcal{J}_X))_\omega := (\Gamma(X;\mathcal{A}^\bullet_X \otimes C^\bullet(\mathcal{J}_X)), \nabla^{can} + \delta + i_\omega)
\]
as the \emph{$\omega$-twist} of $\Gamma(X;\DR(C^\bullet(\mathcal{J}_X))$.

Let
\[
\mathfrak{g}_\DR(\mathcal{J})_\omega :=  \Gamma(X;\DR(C^\bullet(\mathcal{J}_X))[1])_\omega
\]
regarded as a DGLA.  The following theorem is proved in \cite{BGNT1} (Theorem 1 of loc. cit.):

\begin{thm}\label{cor: Def is MC jets}
For any Artin algebra $R$ with maximal ideal $\mathfrak{m}_R$ there is an equivalence of $2$-groupoids
\[
\MC^2(\mathfrak{g}_\DR(\mathcal{J}_X)_\omega\otimes\mathfrak{m}_R) \cong \Def(\mathcal{S})(R)
\]
natural in $R$.
\end{thm}

\section{Formality}\label{review of DTT}
We give a synopsis of the results of \cite{DTT} in the notations of loc. cit. Let $k$ be a field of characteristic
zero. For a $k$-cooperad $\mathcal{C}$ and a complex of $k$-vector spaces $V$ we denote by $\mathbb{F}_\mathcal{C}(V)$
the cofree $\mathcal{C}$-coalgebra on $V$.

We denote by $\mathbf{e_2}$ the operad governing Gerstenhaber algebras. The latter is Koszul, and we denote by
$\dual{\mathbf{e_2}}$ the dual cooperad.

For an associative $k$-algebra $A$ the Hochschild complex $C^\bullet(A)$ has a canonical structure of a brace algebra,
hence a structure of homotopy $\mathbf{e_2}$-algebra. The latter structure is encoded in a differential (i.e. a
coderivation of degree one and square zero) $M \colon  \mathbb{F}_{\dual{\mathbf{e_2}}} (C^\bullet(A)) \to
\mathbb{F}_{\dual{\mathbf{e_2}}} (C^\bullet(A))[1]$.

Suppose from now on that $A$ is regular commutative algebra over a field of characteristic zero (the regularity
assumption is not needed for the constructions). Let $V^\bullet(A) = \Symm^\bullet_A(\Der(A)[-1])$ viewed as a complex
with trivial differential. In this capacity $V^\bullet(A)$ has a canonical structure of an $\mathbf{e_2}$-algebra which
gives rise to the differential $d_{V^\bullet(A)}$ on $\mathbb{F}_{\dual{\mathbf{e_2}}} (V^\bullet(A))$; we have:
$\bc_{\dual{\mathbf{e_2}}} (V^\bullet(A)) = (\mathbb{F}_{\dual{\mathbf{e_2}}} (V^\bullet(A)), d_{V^\bullet(A)})$ (see
\cite{DTT}, Theorem 1 for notations).

In addition, the authors introduce a sub-$\dual{\mathbf{e_2}}$-coalgebra $\Xi(A)$ of both
$\mathbb{F}_{\dual{\mathbf{e_2}}}(C^\bullet(A))$ and $\mathbb{F}_{\dual{\mathbf{e_2}}} (V^\bullet(A))$. We denote by
$\sigma \colon \Xi(A) \to \mathbb{F}_{\dual{\mathbf{e_2}}}(C^\bullet(A))$ and $\iota \colon \Xi(A) \to
\mathbb{F}_{\dual{\mathbf{e_2}}}(V^\bullet(A))$ respective inclusions and identify $\Xi(A)$ with its image under the
latter one. By \cite{DTT}, Proposition 7 the differential $d_{V^\bullet(A)}$ preserves $\Xi(A)$; we denote by
$d_{V^\bullet(A)}$ its restriction to  $\Xi(A)$. By Theorem 3, loc. cit. the inclusion $\sigma$ is a morphism of
complexes. Hence, we have the following diagram in the category of differential graded
$\dual{\mathbf{e_2}}$-coalgebras:
\begin{equation}\label{diag e2 coalg}
(\mathbb{F}_{\dual{\mathbf{e_2}}} (C^\bullet(A)), M) \xleftarrow{\sigma} (\Xi(A), d_{V^\bullet(A)}) \xrightarrow{\iota} \bc_{\dual{\mathbf{e_2}}} (V^\bullet(A))
\end{equation}

Applying the functor $\Omega_\mathbf{e_2}$ (adjoint to the functor $\bc_{\dual{\mathbf{e_2}}}$, see \cite{DTT}, Theorem
1) to \eqref{diag e2 coalg} we obtain the diagram
\begin{multline}\label{diag e2 alg}
\Omega_\mathbf{e_2}(\mathbb{F}_{\dual{\mathbf{e_2}}} (C^\bullet(A)), M) \xleftarrow{\Omega_\mathbf{e_2}(\sigma)} \\
\Omega_\mathbf{e_2}(\Xi(A), d_{V^\bullet(A)}) \xrightarrow{\Omega_\mathbf{e_2}(\iota)}
\Omega_\mathbf{e_2}(\bc_{\dual{\mathbf{e_2}}} (V^\bullet(A)))
\end{multline}
of differential graded $\mathbf{e_2}$-algebras. Let $\nu = \eta_\mathbf{e_2}\circ\Omega_\mathbf{e_2}(\iota)$, where
$\eta_\mathbf{e_2} : \Omega_\mathbf{e_2}(\bc_{\dual{\mathbf{e_2}}} (V^\bullet(A))) \to V^\bullet(A)$ is the counit of
adjunction. Thus, we have the diagram
\begin{equation}\label{diag e2 alg counit}
\Omega_\mathbf{e_2}(\mathbb{F}_{\dual{\mathbf{e_2}}} (C^\bullet(A)), M) \xleftarrow{\Omega_\mathbf{e_2}(\sigma)} \Omega_\mathbf{e_2}(\Xi(A), d_{V^\bullet(A)}) \xrightarrow{\nu} V^\bullet(A)
\end{equation}
of differential graded $\mathbf{e_2}$-algebras.

\begin{thm}[\cite{DTT}, Theorem 4]\label{thm DTT}
The maps $\Omega_\mathbf{e_2}(\sigma)$ and $\nu$ are quasi-isomorphisms.
\end{thm}

Additionally, concerning the DGLA structures relevant to applications to deformation theory, deduced from respective
$\mathbf{e_2}$-algebra structures we have the following result.

\begin{thm}[\cite{DTT}, Theorem 2]\label{e2 to dgla}
The DGLA $\Omega_\mathbf{e_2}(\mathbb{F}_{\dual{\mathbf{e_2}}} (C^\bullet(A)), M)[1]$ and $C^\bullet(A)[1]$ are
canonically $L_\infty$-quasi-isomorphic.
\end{thm}

\begin{cor}[Formality]\label{formality DTT}
The DGLA $C^\bullet(A)[1]$ and $V^\bullet(A)[1]$ are $L_\infty$-quasi-isomorphic.
\end{cor}

\subsection{Some (super-)symmetries}\label{subsection: symmetries}
For applications to deformation theory of algebroid stacks we will need certain equivariance properties of the maps
described in \ref{review of DTT}.

For $a\in A$ let $i_a \colon C^\bullet(A) \to C^\bullet(A)[-1]$ denote the adjoint action (in the sense of the
Gerstenhaber bracket and the identification $A = C^0(A)$). It is given by the formula
\[
i_a D(a_1, \ldots, a_n)=\sum _{i=0}^n (-1)^k D(a_1, \ldots, a_i, a, a_{k+1}, \ldots, a_n) .
\]
The operation $i_a$ extends uniquely to a coderivation of $\mathbb{F}_{\dual{\mathbf{e_2}}} (C^\bullet(A))$; we denote
this extension by $i_a$ as well. Furthermore, the subcoalgebra $\Xi (A)$ is preserved by $i_a$.

Since the operation $i_a$ is a derivation of the cup product as well as of all of the brace operations on
$C^\bullet(A)$ and the homotopy-$\mathbf{e_2}$-algebra structure on $C^\bullet(A)$ given in terms of the cup product
and the brace operations it follows that $i_a$ anti-commutes with the differential $M$. Hence, the coderivation $i_a$
induces a derivation of the differential graded $\mathbf{e_2}$-algebra
$\Omega_\mathbf{e_2}(\mathbb{F}_{\dual{\mathbf{e_2}}} (C^\bullet(A)), M)$ which will be denoted by $i_a$ as well. For
the same reason the DGLA $\Omega_\mathbf{e_2}(\mathbb{F}_{\dual{\mathbf{e_2}}} (C^\bullet(A)), M)[1]$ and
$C^\bullet(A)[1]$ are quasi-isomorphic in a way which commutes with the respective operations $i_a$.

On the other hand, let $i_a \colon V^\bullet (A) \to V^\bullet(A)[-1]$ denote the adjoint action in the sense of the
Schouten bracket and the identification $A = V^0(A)$. The operation $i_a$ extends uniquely to a coderivation of
$\mathbb{F}_{\dual{\mathbf{e_2}}}(V^\bullet(A))$ which anticommutes with the differential $d_{V^\bullet(A)}$ because
$i_a$ is a derivation of the $\mathbf{e_2}$-algebra structure on $V^\bullet(A)$. We denote this coderivation as well as
its unique extension to a derivation of the differential graded $\mathbf{e_2}$-algebra
$\Omega_\mathbf{e_2}(\bc_{\dual{\mathbf{e_2}}} (V^\bullet(A)))$ by $i_a$. The counit map $\eta_\mathbf{e_2} \colon
\Omega_\mathbf{e_2}(\bc_{\dual{\mathbf{e_2}}} (V^\bullet(A))) \to V^\bullet(A)$ commutes with respective operations
$i_a$.

The subcoalgebra $\Xi(A)$ of $\mathbb{F}_{\dual{\mathbf{e_2}}}(C^\bullet(A))$ and
$\mathbb{F}_{\dual{\mathbf{e_2}}}(V^\bullet(A))$ is preserved by the respective operations $i_a$. Moreover, the
restrictions of the two operations to $\Xi(A)$ coincide, i.e. the maps in \eqref{diag e2 coalg} commute with $i_a$ and,
therefore, so do the maps in \eqref{diag e2 alg} and \eqref{diag e2 alg counit}.

\subsection{Deformations of $\mathcal{O}$ and Kontsevich formality}
Suppose that $X$ is a manifold. Let $\mathcal{O}_X$ (respectively, $\mathcal{T}_X$) denote the structure sheaf
(respectively, the sheaf of vector fields). The construction of the diagram localizes on $X$ yielding the diagram of
sheaves of differential graded $\mathbf{e_2}$-algebras
\begin{equation}\label{diag e2 alg counit sheaves}
\Omega_\mathbf{e_2}(\mathbb{F}_{\dual{\mathbf{e_2}}} (C^\bullet(\mathcal{O}_X)), M) \xleftarrow{\Omega_\mathbf{e_2}(\sigma)} \Omega_\mathbf{e_2}(\Xi(\mathcal{O}_X), d_{V^\bullet(\mathcal{O}_X)}) \xrightarrow{\nu} V^\bullet(\mathcal{O}_X) ,
\end{equation}
where $C^\bullet(\mathcal{O}_X)$ denotes the sheaf of multidifferential operators and $V^\bullet(\mathcal{O}_X) :=
\Symm^\bullet_{\mathcal{O}_X}(\mathcal{T}_X[-1])$ denotes the sheaf of multivector fields. Theorem \ref{thm DTT}
extends easily to this case stating that the morphisms $\Omega_\mathbf{e_2}(\sigma)$ and $\nu$ in \eqref{diag e2 alg
counit sheaves} are quasi-isomorphisms of sheaves of differential graded $\mathbf{e_2}$-algebras.


\section{Formality for the algebroid Hochschild complex}\label{for}
\subsection{A version of \cite{DTT} for jets}\label{subsection: DTT for jets}
Let $C^\bullet(\mathcal{J}_X)$ denote sheaf of continuous (with respect to the adic topology)
$\mathcal{O}_X$-multilinear Hochschild cochains on $\mathcal{J}_X$. Let $V^\bullet(\mathcal{J}_X) =
\Symm^\bullet_{\mathcal{J}_X}(\Der^{cont}_{\mathcal{O}_X}(\mathcal{J}_X)[-1])$.

Working now in the category of graded $\mathcal{O}_X$-modules we have the diagram
\begin{equation}\label{diag e2 alg jets}
\Omega_\mathbf{e_2}(\mathbb{F}_{\dual{\mathbf{e_2}}}(C^\bullet(\mathcal{J}_X)), M) \xleftarrow{\Omega_\mathbf{e_2}(\sigma)} \\ \Omega_\mathbf{e_2}(\Xi(\mathcal{J}_X), d_{V^\bullet(\mathcal{J}_X)}) \xrightarrow{\nu} V^\bullet(\mathcal{J}_X)
\end{equation}
of sheaves of differential graded $\mathcal{O}_X$-$\mathbf{e_2}$-algebras. Theorem \ref{thm DTT} extends easily to this
situation: the morphisms $\Omega_\mathbf{e_2}(\sigma)$ and $\nu$ in \eqref{diag e2 alg jets} are quasi-isomorphisms.
The sheaves of DGLA $\Omega_\mathbf{e_2}(\mathbb{F}_{\dual{\mathbf{e_2}}}(C^\bullet(\mathcal{J}_X)), M)[1]$ and $
C^\bullet(\mathcal{J}_X)[1]$ are canonically $L_\infty$-quasi-isomorphic.

The canonical flat connection $\nabla^{can}$ on $\mathcal{J}_X$ induces a flat connection which we denote
$\nabla^{can}$ as well on each of the objects in the diagram \eqref{diag e2 alg jets}. Moreover, the maps
$\Omega_\mathbf{e_2}(\sigma)$ and $\nu$ are flat with respect to $\nabla^{can}$ hence induce the maps of respective de
Rham complexes
\begin{multline}\label{diag e2 alg jets de Rham}
\DR(\Omega_\mathbf{e_2}(\mathbb{F}_{\dual{\mathbf{e_2}}}(C^\bullet(\mathcal{J}_X)), M))
\xleftarrow{\DR(\Omega_\mathbf{e_2}(\sigma))} \\ \DR(\Omega_\mathbf{e_2}(\Xi(\mathcal{J}_X),
d_{V^\bullet(\mathcal{J}_X)})) \xrightarrow{\DR(\nu)} \DR(V^\bullet(\mathcal{J}_X))
\end{multline}
where, for $(K^\bullet, d)$ one of the objects in \eqref{diag e2 alg jets} we denote by $\DR(K^\bullet,d)$ the total
complex of the double complex $(\mathcal{A}^\bullet_X\otimes K^\bullet, d, \nabla^{can})$. All objects in the diagram
\eqref{diag e2 alg jets de Rham} have canonical structures of differential graded $\mathbf{e_2}$-algebras and the maps
are morphisms thereof.

The DGLA $\Omega_\mathbf{e_2}(\mathbb{F}_{\dual{\mathbf{e_2}}}(C^\bullet(\mathcal{J}_X)), M)[1]$ and
$C^\bullet(\mathcal{J}_X)[1]$ are canonically $L_\infty$-quasi-isomorphic in a way compatible with $\nabla^{can}$.
Hence, the DGLA $\DR(\Omega_\mathbf{e_2}(\mathbb{F}_{\dual{\mathbf{e_2}}}(C^\bullet(\mathcal{J}_X)), M)[1])$ and
$\DR(C^\bullet(\mathcal{J}_X)[1])$ are canonically $L_\infty$-quasi-isomorphic.

\subsection{A version of \cite{DTT} for jets with a twist}\label{subsection: jets with a twist}
Suppose that $\omega\in\Gamma(X;\mathcal{A}^2_X\otimes\mathcal{J}_X/\mathcal{O}_X)$ satisfies $\nabla^{can}\omega = 0$.

For each of the objects in \eqref{diag e2 alg jets de Rham} we denote by $i_\omega$ the operation which is induced by
the one described in \ref{subsection: symmetries} and the wedge product on $\mathcal{A}^\bullet_X$. Thus, for each
differential graded $\mathbf{e_2}$-algebra $(N^\bullet, d)$ in \eqref{diag e2 alg jets de Rham} we have a derivation of
degree one and square zero $i_\omega$ which anticommutes with $d$ and we denote by $(N^\bullet,d)_\omega$ the
\emph{$\omega$-twist} of $(N^\bullet, d)$, i.e. the differential graded $\mathbf{e_2}$-algebra $(N^\bullet, d +
i_\omega)$. Since the morphisms in \eqref{diag e2 alg jets de Rham} commute with the respective operations $i_\omega$,
they give rise to morphisms of respective $\omega$-twists
\begin{multline}\label{diag e2 alg jets de Rham twisted}
\DR(\Omega_\mathbf{e_2}(\mathbb{F}_{\dual{\mathbf{e_2}}}(C^\bullet(\mathcal{J}_X)), M))_\omega
\xleftarrow{\DR(\Omega_\mathbf{e_2}(\sigma))} \\ \DR(\Omega_\mathbf{e_2}(\Xi(\mathcal{J}_X),
d_{V^\bullet(\mathcal{J}_X)}))_\omega \xrightarrow{\DR(\nu)} \DR(V^\bullet(\mathcal{J}_X))_\omega .
\end{multline}

Let $F_\bullet\mathcal{A}^\bullet_X$ denote the stupid filtration: $F_i\mathcal{A}^\bullet_X = \mathcal{A}^{\geq
-i}_X$. The filtration $F_\bullet\mathcal{A}^\bullet_X$ induces a filtration denoted $F_\bullet\DR(K^\bullet,d)_\omega$
for each object $(K^\bullet, d)$ of \eqref{diag e2 alg jets} defined by $F_i\DR(K^\bullet,d)_\omega =
F_i\mathcal{A}^\bullet_X\otimes K^\bullet$. As is easy to see, the associated graded complex is given by
\begin{equation}\label{Gr stupid}
Gr_{-p}\DR(K^\bullet,d)_\omega = (\mathcal{A}^p_X\otimes K^\bullet , \id\otimes d) .
\end{equation}
It is clear that the morphisms $\DR(\Omega_\mathbf{e_2}(\sigma))$ and $\DR(\nu)$ are filtered with respect to
$F_\bullet$.

\begin{thm}\label{diag e2 coalg jets de Rham twisted are filtered quisms}
The morphisms in \eqref{diag e2 alg jets de Rham twisted} are filtered quasi-isomorphisms, i.e. the maps
$Gr_i\DR(\Omega_\mathbf{e_2}(\sigma))$ and $Gr_i\DR(\nu)$ are quasi-isomorphisms for all $i \in \mathbb{Z}$.
\end{thm}
\begin{proof}
We consider the case of $\DR(\Omega_\mathbf{e_2}(\sigma))$ leaving $Gr_i\DR(\nu)$ to the reader.

The map $Gr_{-p}\DR(\Omega_\mathbf{e_2}(\sigma))$ induced by $\DR(\Omega_\mathbf{e_2}(\sigma))$ on the respective
associated graded objects in degree $-p$ is equal to the map of complexes
\begin{equation}\label{Gr DR}
\id\otimes\Omega_\mathbf{e_2}(\sigma) \colon \mathcal{A}^p_X\otimes\Omega_\mathbf{e_2}(\Xi(\mathcal{J}_X), d_{V^\bullet(\mathcal{J}_X)}) \to \mathcal{A}^p_X\otimes\Omega_\mathbf{e_2}(\mathbb{F}_{\dual{\mathbf{e_2}}}(C^\bullet(\mathcal{J}_X)), M) .
\end{equation}
The map $\sigma$ is a quasi-isomorphism by Theorem \ref{thm DTT}, therefore so is $\Omega_\mathbf{e_2}(\sigma)$. Since
$\mathcal{A}^p_X$ is flat over $\mathcal{O}_X$, the map \eqref{Gr DR} is a quasi-isomorphism.
\end{proof}

\begin{cor}\label{cor: e2 coalg quisms}
The maps $\DR(\Omega_\mathbf{e_2}(\sigma))$ and $\DR(\nu)$ in \eqref{diag e2 alg jets de Rham twisted} are
quasi-isomorphisms of sheaves of differential graded $\mathbf{e_2}$-algebras.
\end{cor}

Additionally, the DGLA $\DR(\Omega_\mathbf{e_2}(\mathbb{F}_{\dual{\mathbf{e_2}}}(C^\bullet(\mathcal{J}_X)), M)[1])$ and
$\DR(C^\bullet(\mathcal{J}_X)[1])$ are canonically $L_\infty$-quasi-isomorphic in a way which commutes with the
respective operations $i_\omega$ which implies that the respective $\omega$-twists
$\DR(\Omega_\mathbf{e_2}(\mathbb{F}_{\dual{\mathbf{e_2}}}(C^\bullet(\mathcal{J}_X)), M)[1])_\omega$ and
$\DR(C^\bullet(\mathcal{J}_X)[1])_\omega$ are canonically $L_\infty$-quasi-isomorphic.

\subsection{$L_\infty$-structures on multivectors}\label{subsection: structures on multivectors}
The canonical pairing $\ip \colon \mathcal{A}^1_X \otimes \mathcal{T}_X \to \mathcal{O}_X$ extends to the pairing
\[
\ip \colon \mathcal{A}^1_X \otimes V^\bullet(\mathcal{O}_X) \to V^\bullet(\mathcal{O}_X)[-1]
\]

For $k \geq 1$, $\omega = \alpha_1\wedge\ldots\wedge\alpha_k$, $\alpha_i\in\mathcal{A}^1_X$, $i = 1,\ldots,k$, let
\[
\Phi(\omega) \colon \Symm^k V^\bullet(\mathcal{O}_X)[2] \to V^\bullet(\mathcal{O}_X)[k]
\]
denote the map given by the formula
\begin{multline*}
\Phi(\omega)(\pi_1,\ldots,\pi_k) = (-1)^{(k-1)(|\pi_1|-1)+\ldots +2|(\pi_{k-3}|-1)+(|\pi_{k-2}|-1)}\times \\
\sum_\sigma \sgn(\sigma) \langle\alpha_1,\pi_{\sigma(1)}\rangle\wedge \dots
\wedge\langle\alpha_k,\pi_{\sigma(k)}\rangle ,
\end{multline*}
where $|\pi|=l$ for $\pi \in V^l({\mathcal O}_X)$. For $\alpha \in \mathcal{O}_X$ let $\Phi(\alpha) = \alpha \in
V^0(\mathcal{O}_X)$.

Recall that a graded vector space $W$ gives rise to the graded Lie algebra $\Der(\coComm(W[1]))$. An element
$\gamma\in\Der(\coComm(W[1]))$ of degree one which satisfies $[\gamma,\gamma] = 0$ defines a structure of an $L_\infty$-algebra on $W$. Such a $\gamma$ determines a differential $\partial_\gamma := [\gamma, . ]$ on $\Der(\coComm(W[1]))$,
such that $(\Der(\coComm(W[1])), \partial_\gamma)$ is a differential graded Lie algebra. If $\mathfrak{g}$ is a graded
Lie algebra and $\gamma$ is the element of $\Der(\coComm(\mathfrak{g}[1]))$ corresponding to the bracket on
$\mathfrak{g}$, then $(\Der(\coComm(\mathfrak{g}[1])), \partial_\gamma)$ is equal to the shifted Chevalley cochain
complex $C^\bullet(\mathfrak{g};\mathfrak{g})[1]$.

In what follows we consider the (shifted) de Rham complex $\mathcal{A}^\bullet_X[2]$ as a differential graded Lie
algebra with the trivial bracket.
\begin{lemma}\label{de Rham to Chevalley}
The map $\omega \mapsto \Phi(\omega)$ defines a morphism of sheaves of differential graded Lie algebras
\begin{equation}\label{DR to def}
\Phi \colon \mathcal{A}^\bullet_X[2] \to C^\bullet(V^\bullet(\mathcal{O}_X)[1]; V^\bullet(\mathcal{O}_X)[1])[1] .
\end{equation}
\end{lemma}
\begin{proof}
Recall the explicit formulas for the Schouten bracket. If $f$ and $g$ are functions and $X_i$, $Y_j$ are vector fields,
then
\begin{multline*}
[fX_1\ldots X_k,gY_1\ldots Y_l] = \sum_i (-1)^{k-i}fX_k(g)X_1\ldots {\widehat{X_i}}\ldots X_kY_1\ldots Y_l + \\
\sum _j(-1)^j Y_j(f)gX_1\ldots X_kY_1\ldots {\widehat{Y_j}}\ldots Y_l + \\
\sum_{i,j}(-1)^{i+j}fgX_1\ldots {\widehat{X_i}}\ldots X_kY_1\ldots {\widehat{Y_j}}\ldots Y_l
\end{multline*}
Note that for a one-form $\omega$ and for vector fields $X$ and $Y$
\begin{equation}\label{eq:phixy}
\langle \omega,[X,Y]\rangle-\langle [\omega,X],Y\rangle-\langle X,[\omega,Y]\rangle=\Phi(d\omega)(X,Y)
\end{equation}
From the two formulas above we deduce by an explicit computation that
\[
\langle \omega,[\pi,\rho]\rangle-\langle [\omega,\pi],\rho\rangle-(-1)^{|\pi|-1}\langle \pi,[\omega,\rho]\rangle=
(-1)^{|\pi|-1}\Phi(d\omega)(\pi,\rho)
\]
Note that Lie algebra cochains are invariant under the symmetric group acting by permutations multiplied by signs that
are computed by the following rule: a permutation of $\pi_i$ and $\pi_j$ contributes a factor $(-1)^{|\pi_i||\pi_j|}.$
We use the explicit formula for the bracket on the Lie algebra complex.
\[
[\Phi,\Psi]=\Phi\circ \Psi-(-1)^{|\Phi||\Psi|}\Psi\circ\Phi
\]
\[
(\Phi\circ \Psi)(\pi_1,\ldots,\pi_{k+l-1})=\sum_{I,J}\epsilon(I,J)\Phi(\Psi(\pi_{i_1},\ldots,\pi_{i_k}),\pi_{j_1},\ldots,\pi_{j_{l-1}})
\]
Here $I=\{i_1,\ldots,i_k\};$ $J=\{j_1,\ldots,j_{l-1}\};$ $i_1<\ldots<i_k;$ $j_1<\ldots<j_{l-1};$ $I\coprod
J=\{1,\ldots,k+l-1\};$ the sign $\epsilon(I,J)$ is computed by the same sign rule as above. The differential is given
by the formula
\[\partial \Phi=[m,\Phi]
\]
where $m(\pi,\rho)=(-1)^{|\pi|-1}[\pi,\rho].$ Let $\alpha=\alpha_1\ldots\alpha_k$ and $\beta=\beta_1\ldots\beta_l.$ We
see from the above that both cochains $\Phi(\alpha)\circ \Phi(\beta)$ and $\Phi(\beta)\circ \Phi(\alpha)$ are
antisymmetrizations with respect to $\alpha_i$ and $\beta_j$ of the sums
\[
\sum_{I,J,p} \pm \langle \alpha_1\beta_1,\pi _p\rangle \langle \alpha_2,\pi_{i_1}\rangle \ldots \langle \alpha_k,\pi_{i_{k-1}}\rangle \langle \beta_2,\pi_{j_1}\rangle \ldots \langle \beta_l,\pi_{j_{l-1}}\rangle
\]
over all partitions $\{1,\ldots,k+l-1\}=I\coprod J \coprod \{p\}$ where $i_1<\ldots<i_{k-1}$ and $j_1<\ldots<j_{l-1};$
here $\langle \alpha\beta,\pi\rangle=\langle \alpha, \langle \beta,\pi\rangle\rangle.$ After checking the signs, we
conclude that $[\Phi(\alpha), \Phi(\beta)]=0.$ Also, from the definition of the differential, we see that $\partial
\Phi(\alpha)(\pi_1, \ldots, \pi _{k+1})$ is the antisymmetrizations with respect to $\alpha_i$ and $\beta_j$ of the sum
\begin{multline*}
\sum_{i<j}\pm (\langle \alpha_1,[\pi_i,\pi_j]\rangle-\langle[\alpha_1, \pi_i],\pi_j\rangle-(-1)^{|\pi_i|-1}[\pi_i, \langle \alpha_1,\pi_j\rangle])\cdot \\
\langle \alpha_2,\pi_{1}\rangle\ldots {\langle \alpha_i,\pi_{i-1}\rangle}\langle \alpha_{i+1},\pi_{i+1}\rangle\ldots
{\langle \alpha_{j-1},\pi_{j-1}\rangle}{\langle \alpha_j,\pi_{j+1}\rangle}\langle \alpha_k,\pi_{k+1}\rangle
\end{multline*}
We conclude from this and \eqref{eq:phixy}  that $\partial \Phi(\alpha)=\Phi(d\alpha)$.
\end{proof}

Thus, according to Lemma \ref{de Rham to Chevalley}, a closed $3$-form $H$ on $X$ gives rise to a Maurer-Cartan element
$\Phi(H)$ in $\Gamma(X;C^\bullet(V^\bullet(\mathcal{O}_X)[1]; V^\bullet(\mathcal{O}_X)[1])[1])$, hence a structure of
an $L_\infty$-algebra on $V^\bullet(\mathcal{O}_X)[1]$ which has the trivial differential (the unary operation), the
binary operation equal to the Schouten-Nijenhuis bracket, the ternary operation given by $\Phi(H)$, and all higher
operations equal to zero. Moreover, cohomologous closed $3$-forms give rise to gauge equivalent Maurer-Cartan elements,
hence to $L_\infty$-isomorphic $L_\infty$-structures.

\begin{notation}
For a closed $3$-form $H$ on $X$ we denote the corresponding $L_\infty$-algebra structure on
$V^\bullet(\mathcal{O}_X)[1]$ by $V^\bullet(\mathcal{O}_X)[1]_H$. Let
\[
\mathfrak{s}(\mathcal{O}_X)_H := \Gamma(X;V^\bullet(\mathcal{O}_X)[1])_H .
\]
\end{notation}

\subsection{$L_\infty$-structures on multivectors via formal geometry}\label{subsection: structures on multivectors formal}
In order to relate the results of \ref{subsection: jets with a twist} with those of \ref{subsection: structures on
multivectors} we consider the analog of the latter for jets.

Let $\widehat{\Omega}^k_{\mathcal{J}/\mathcal{O}} := \mathcal{J}_X(\mathcal{A}^k_X)$, the sheaf of jets of differential
$k$-forms on $X$. Let $\dR$ denote the ($\mathcal{O}_X$-linear) differential in
$\widehat{\Omega}^\bullet_{\mathcal{J}/\mathcal{O}}$ induced by the de Rham differential in $\mathcal{A}^\bullet_X$.
The differential $\dR$ is horizontal with respect to the canonical flat connection $\nabla^{can}$ on
$\widehat{\Omega}^\bullet_{\mathcal{J}/\mathcal{O}}$, hence we have the double complex
$(\mathcal{A}^\bullet_X\otimes\widehat{\Omega}^\bullet_{\mathcal{J}/\mathcal{O}}, \nabla^{can},\id\otimes\dR)$ whose
total complex is denoted $\DR(\widehat{\Omega}^\bullet_{\mathcal{J}/\mathcal{O}})$.

Let $\vac \colon \mathcal{O}_X \to \mathcal{J}_X$ denote the unit map (not to be confused with the map $j^\infty$); it
is an isomorphism onto the kernel of $\dR : \mathcal{J}_X \to \widehat{\Omega}^1_{\mathcal{J}/\mathcal{O}}$ and
therefore defines the morphism of complexes $\vac \colon \mathcal{O}_X \to
\widehat{\Omega}^\bullet_{\mathcal{J}/\mathcal{O}}$ which is a quasi-isomorphism. The map $\vac$ is horizontal with
respect to the canonical flat connections on $\mathcal{O}_X$ and $\mathcal{J}_X$ (respectively,
$\widehat{\Omega}^\bullet_{\mathcal{J}/\mathcal{O}}$), therefore we have the induced map of respective de Rham
complexes $\DR(\vac) \colon \mathcal{A}^\bullet_X \to \DR(\mathcal{J}_X)$ (respectively, $\DR(\vac) \colon
\mathcal{A}^\bullet_X \to \DR(\widehat{\Omega}^\bullet_{\mathcal{J}/\mathcal{O}})$, a quasi-isomorphism).

Let $C^\bullet(\mathfrak{g}(\mathcal{J}_X); \mathfrak{g}(\mathcal{J}_X))$ denote the complex of continuous
$\mathcal{O}_X$-multilinear cochains. The map of differential graded Lie algebras
\begin{equation}\label{DR to def jet}
\widehat{\Phi} \colon \widehat{\Omega}^\bullet_{\mathcal{J}/\mathcal{O}}[2] \to
C^\bullet(V^\bullet(\mathcal{J}_X)[1]; V^\bullet(\mathcal{J}_X)[1])[1]
\end{equation}
defined in the same way as \eqref{DR to def} is horizontal with respect to the canonical flat connection $\nabla^{can}$
and induces the map
\begin{equation}\label{DR to def de Rham}
\DR(\widehat{\Phi}) \colon \DR(\widehat{\Omega}^\bullet_{\mathcal{J}/\mathcal{O}})[2] \to \DR((C^\bullet(V^\bullet(\mathcal{J}_X)[1]; V^\bullet(\mathcal{J}_X)[1])[1])
\end{equation}
There is a canonical morphism of sheaves of differential graded Lie algebras
\begin{equation}\label{DR Der to Der DR}
\DR(C^\bullet(V^\bullet(\mathcal{J}_X)[1]; V^\bullet(\mathcal{J}_X)[1])[1]) \to C^\bullet(\DR(V^\bullet(\mathcal{J}_X)[1]); \DR(V^\bullet(\mathcal{J}_X)[1]))[1]
\end{equation}
Therefore, a degree three cocycle in $\Gamma(X;\DR(\widehat{\Omega}^\bullet_{\mathcal{J}/\mathcal{O}}))$ determines an
$L_\infty$-structure on $\DR(V^\bullet(\mathcal{J}_X)[1])$ and cohomologous cocycles determine $L_\infty$-isomorphic
structures.

\begin{notation}
For a section $B \in \Gamma(X;\mathcal{A}^2_X\otimes\mathcal{J}_X)$ we denote by $\overline{B}$ it's image in
$\Gamma(X;\mathcal{A}^2_X\otimes\mathcal{J}_X/\mathcal{O}_X)$.
\end{notation}

\begin{lemma}\label{H exists}
If $B \in \Gamma(X;\mathcal{A}^2_X\otimes\mathcal{J}_X)$ satisfies $\nabla^{can}\overline{B} = 0$, then
\begin{enumerate}
\item $\dR B$ is a (degree three) cocycle in $\Gamma(X;\DR(\widehat{\Omega}^\bullet_{\mathcal{J}/\mathcal{O}}))$;

\item there exist a unique $H \in \Gamma(X;\mathcal{A}^3_X)$ such that $dH=0$ and $\DR(\vac)(H) = \nabla^{can}B$.
\end{enumerate}
\end{lemma}
\begin{proof}
For the first claim it suffices to show that $\nabla^{can}B = 0$. This follows from the assumption that
$\nabla^{can}\overline{B} = 0$ and the fact that $\dR : \mathcal{A}^\bullet_X\otimes\mathcal{J}_X \to
\mathcal{A}^\bullet_X\otimes\widehat{\Omega}^1_{\mathcal{J}/\mathcal{O}}$ factors through
$\mathcal{A}^\bullet_X\otimes\mathcal{J}_X/\mathcal{O}_X$.

We have: $\dR\nabla^{can} B = \nabla^{can}\dR B = 0$. Therefore, $\nabla^{can} B$ is in the image of $\DR(\vac) :
\Gamma(X;\mathcal{A}^3_X) \to \Gamma(X;\mathcal{A}^3_X\otimes\mathcal{J}_X)$ which is injective, whence the existence
and uniqueness of $H$. Since $\DR(\vac)$ is a morphism of complexes it follows that $H$ is closed.
\end{proof}

Suppose that $B \in \Gamma(X;\mathcal{A}^2_X\otimes\mathcal{J}_X)$ satisfies $\nabla^{can}\overline{B} = 0$. Then, the
differential graded Lie algebra $\DR(\mathfrak{g}(\mathcal{J}_X))_{\overline{B}}$ (the $\overline{B}$-twist of
$\DR(\mathfrak{g}(\mathcal{J}_X))$) is defined. On the other hand, due to Lemma \ref{H exists}, \eqref{DR to def de
Rham} and \eqref{DR Der to Der DR}, $\dR B$ gives rise to an $L_\infty$-structure on
$\DR(V^\bullet(\mathcal{J}_X)[1])$.

\begin{lemma}\label{lemma: dB gives twist}
The $L_\infty$-structure induced by $\dR B$ is that of a differential graded Lie algebra equal to
$\DR(V^\bullet(\mathcal{J}_X)[1])_{\overline{B}}$.
\end{lemma}
\begin{proof}
Left to the reader.
\end{proof}

\begin{notation}
For a $3$-cocycle $\omega \in \Gamma(X;\DR(\widehat{\Omega}^\bullet_{\mathcal{J}/\mathcal{O}}))$ we will denote by
$\DR(V^\bullet(\mathcal{J}_X)[1])_\omega$ the $L_\infty$-algebra obtained from $\omega$ via \eqref{DR to def de Rham}
and \eqref{DR Der to Der DR}. Let
\[
\mathfrak{s}_\DR(\mathcal{J}_X)_\omega := \Gamma(X;\DR(V^\bullet(\mathcal{J}_X)[1]))_\omega .
\]

\end{notation}
\begin{remark}
Lemma \ref{lemma: dB gives twist} shows that this notation is unambiguous with reference to the previously introduced
notation for the twist. In the notations introduced above, $\dR B$ is the image of $\overline{B}$ under the
\emph{injective} map $\Gamma(X; \mathcal{A}^2_X \otimes \mathcal{J}_X/\mathcal{O}_X) \to \Gamma(X; \mathcal{A}^2_X
\otimes\widehat{\Omega}^1_{\mathcal{J}/\mathcal{O}})$ which factors $\dR$ and ``allows'' us to ``identify''
$\overline{B}$ with $\dR B$.
\end{remark}

\begin{thm}\label{thm: hoch jet is schouten twist}
Suppose that $B \in \Gamma(X;\mathcal{A}^2_X\otimes\mathcal{J}_X)$ satisfies $\nabla^{can}\overline{B} = 0$. Let $H \in
\Gamma(X;\mathcal{A}^3_X)$ denote the unique $3$-form such that $\DR(\vac)(H) = \nabla^{can}B$ (cf. Lemma \ref{H
exists}). Then, the $L_\infty$-algebras $\mathfrak{g}_\DR(\mathcal{J}_X)_{\overline{B}}$ and
$\mathfrak{s}(\mathcal{O}_X)_H$ are $L_\infty$-quasi-isomorphic.
\end{thm}
\begin{proof}
The map $j^\infty \colon V^\bullet(\mathcal{O}_X) \to V^\bullet(\mathcal{J}_X)$ induces a quasi-isomorphism of sheaves
of DGLA
\begin{equation}\label{j-infty}
j^\infty \colon V^\bullet(\mathcal{O}_X)[1] \to \DR(V^\bullet(\mathcal{J}_X)[1]) .
\end{equation}
Suppose that $H$ is a closed $3$-form on $X$. Then, the map \eqref{j-infty} is a quasi-isomorphism of sheaves of
$L_\infty$-algebras
\[
j^\infty \colon V^\bullet(\mathcal{O}_X)[1]_H \to \DR(V^\bullet(\mathcal{J}_X)[1])_{\DR(\vac)(H)} .
\]
Passing to global section we obtain the quasi-isomorphism of $L_\infty$-algebras
\begin{equation}\label{j-infty H}
j^\infty \colon \mathfrak{s}(\mathcal{O}_X)_H \to \mathfrak{s}_\DR(\mathcal{J}_X)_{\DR(\vac)(H)} .\end{equation}

By assumption, $B$ provides a homology between $\dR B$ and $\nabla^{can} B = \DR(\vac)(H)$. Therefore, we have the
corresponding $L_\infty$-quasi-isomorphism
\begin{equation}
\DR(V^\bullet(\mathcal{J}_X)[1])_{\DR(\vac)(H)} \stackrel{L_\infty}{\cong} \DR(V^\bullet(\mathcal{J}_X)[1])_{\dR B} = \DR(V^\bullet(\mathcal{J}_X)[1])_{\overline{B}}
\end{equation}
(the second equality is due to Lemma \ref{lemma: dB gives twist}).

According to Corollary \ref{cor: e2 coalg quisms} the sheaf of DGLA  $\DR(V^\bullet(\mathcal{J}_X)[1])_{\overline{B}}$
is $L_\infty$-quasi-isomorphic to the DGLA deduced form the differential graded $\mathbf{e_2}$-algebra
$\DR(\Omega_\mathbf{e_2}(\mathbb{F}_{\dual{\mathbf{e_2}}}(C^\bullet(\mathcal{J}_X)), M))_{\overline{B}}$. The latter
DGLA is $L_\infty$-quasi-isomorphic to $\DR(C^\bullet(\mathcal{J}_X)[1])_{\overline{B}}$.

Passing to global sections we conclude that $\mathfrak{s}_\DR(\mathcal{J}_X)_{\DR(\vac)(H)}$ and
$\mathfrak{g}_\DR(\mathcal{J}_X)_{\overline{B}}$ are $L_\infty$-quasi-isomorphic. Together with \eqref{j-infty H} this
implies the claim.
\end{proof}

\section{Application to deformation theory}\label{deform}

\begin{thm}\label{main}
Suppose that $\mathcal{S}$ is a twisted form of $\mathcal{O}_X$ (\ref{subsection: twisted forms}). Let $H$ be a closed
$3$-form on $X$ which represents $[\mathcal{S}]_{dR} \in H^3_{dR}(X)$. For any Artin algebra $R$ with maximal ideal
$\mathfrak{m}_R$ there is an equivalence of $2$-groupoids
\[
\MC^2(\mathfrak{s}(\mathcal{O}_X)_H\otimes\mathfrak{m}_R) \cong \Def(\mathcal{S})(R)
\]
natural in $R$.
\end{thm}
\begin{proof}
Since cohomologous $3$-forms give rise to $L_\infty$-quasi-isomorphic $L_\infty$-algebras we may assume, possibly
replacing $H$ by another representative of $[\mathcal{S}]_{dR}$, that there exists
$B\in\Gamma(X;\mathcal{A}_X^2\otimes\mathcal{J}_X)$ such that $\overline{B}$ represents $[\mathcal{S}]$ and
$\nabla^{can}B = \DR(\vac)(H)$. By Theorem \ref{thm: hoch jet is schouten twist} $\mathfrak{s}(\mathcal{O}_X)_H$ is
$L_\infty$-quasi-isomorphic to $\mathfrak{g}_\DR(\mathcal{J}_X)_{\overline{B}}$. By the Theorem \ref{thm: quism invariance of mc}
we have a homotopy equivalence of nerves of $2$-groupoids $\gamma_{\bullet}(\mathfrak{s}(\mathcal{O}_X)_H\otimes\mathfrak{m}_R) \cong
\gamma_{\bullet}(\mathfrak{g}_\DR(\mathcal{J}_X)_{\overline{B}}\otimes\mathfrak{m}_R) $.
Therefore, there are equivalences
\[
\MC^2(\mathfrak{s}(\mathcal{O}_X)_H\otimes\mathfrak{m}_R) \cong
\MC^2(\mathfrak{g}_\DR(\mathcal{J}_X)_{\overline{B}}\otimes\mathfrak{m}_R) \cong
\Def(\mathcal{S})(R) ,
\]
the second one being that of Theorem \ref{cor: Def is MC jets}.
\end{proof}

\begin{remark}
In particular, the isomorphism classes of formal deformations of $\mathcal S$ are in a bijective correspondence with
equivalence classes of Maurer-Cartan elements of the $L_\infty$-algebra $\mathfrak{s}_\DR(\mathcal{O}_X)_H
\widehat{\otimes} t\cdot\mathbb{C}[[t]]$. These are the formal \emph{twisted Poisson structures} in the terminology of
\cite{SW}, i.e. the formal series $\pi =\sum_{k=1}^\infty t^k\pi_k$, $\pi_k \in \Gamma(X; \bigwedge^2\mathcal{T}_X)$,
satisfying the equation
\[
[\pi,\pi] = \Phi(H)(\pi,\pi,\pi).
\]
A construction of an algebroid stack associated to a twisted Poisson structure was proposed by P.~\v Severa in
\cite{S}.
\end{remark}


\begin{thebibliography}{10}

\bibitem{BGNT}
P.~Bressler, A.~Gorokhovsky, R.~Nest, and B.~Tsygan.
\newblock Deformation quantization of gerbes.
\newblock {\em Adv. Math.}, 214(1):230--266, 2007.

\bibitem{BGNT1}
P.~Bressler, A.~Gorokhovsky, R.~Nest, and B.~Tsygan.
\newblock Deformations of gerbes on smooth manifolds.
\newblock In {\em K-theory and noncommutative geometry}. European Mathematical
  Society, 2008.

\bibitem{DAP}
A.~D'Agnolo and P.~Polesello.
\newblock Stacks of twisted modules and integral transforms.
\newblock In {\em Geometric aspects of Dwork theory. Vol. I, II}, pages
  463--507. Walter de Gruyter GmbH \& Co. KG, Berlin, 2004.

\bibitem{Del}
P.~Deligne.
\newblock {Letter to L.~Breen}, 1994.

\bibitem{DTT}
V.~Dolgushev, D.~Tamarkin, and B.~Tsygan.
\newblock The homotopy {G}erstenhaber algebra of {H}ochschild cochains of a
  regular algebra is formal.
\newblock {\em J. Noncommut. Geom.}, 1(1):1--25, 2007.

\bibitem{duskin}
J.~W. Duskin.
\newblock Simplicial matrices and the nerves of weak {$n$}-categories. {I}.
  {N}erves of bicategories.
\newblock {\em Theory Appl. Categ.}, 9:198--308 (electronic), 2001/02.
\newblock CT2000 Conference (Como).

\bibitem{G1}
E.~Getzler.
\newblock {Lie theory for nilpotent L-infinity algebras}.
\newblock \emph{Ann. of Math.}, to appear, arXiv:math.AT/0404003.

\bibitem{G2}
E.~Getzler.
\newblock A {D}arboux theorem for {H}amiltonian operators in the formal
  calculus of variations.
\newblock {\em Duke Math. J.}, 111(3):535--560, 2002.

\bibitem{KS}
M.~Kashiwara and P.~Schapira.
\newblock {\em Categories and sheaves}, volume 332 of {\em Grundlehren der
  Mathematischen Wissenschaften [Fundamental Principles of Mathematical
  Sciences]}.
\newblock Springer-Verlag, Berlin, 2006.

\bibitem{K1}
M.~Kontsevich.
\newblock Formality conjecture.
\newblock In {\em Deformation theory and symplectic geometry ({A}scona, 1996)},
  volume~20 of {\em Math. Phys. Stud.}, pages 139--156. Kluwer Acad. Publ.,
  Dordrecht, 1997.

\bibitem{K}
M.~Kontsevich.
\newblock Deformation quantization of {P}oisson manifolds.
\newblock {\em Lett. Math. Phys.}, 66(3):157--216, 2003.

\bibitem{S}
P.~{\v{S}}evera.
\newblock Quantization of {P}oisson families and of twisted {P}oisson
  structures.
\newblock {\em Lett. Math. Phys.}, 63(2):105--113, 2003.

\bibitem{SW}
P.~{\v{S}}evera and A.~Weinstein.
\newblock Poisson geometry with a 3-form background.
\newblock {\em Progr. Theoret. Phys. Suppl.}, (144):145--154, 2001.
\newblock Noncommutative geometry and string theory (Yokohama, 2001).

\bibitem{T}
D.~Tamarkin.
\newblock Another proof of {M}. {K}ontsevich formality theorem.
\newblock preprint, arXiv:math/9803025v4 [math.QA].

\end{thebibliography}
\end{document}